\documentclass[11pt,a4paper]{article}

\usepackage[headinclude]{typearea}
\usepackage[english]{babel}           
\usepackage[T1]{fontenc}
\usepackage{scrtime}
\usepackage{amsmath,amsthm,amsfonts,amssymb}
\usepackage {color}                                  
\usepackage {pifont}                                 
\usepackage {multicol,multirow}                               
\usepackage[numbers]{natbib}                                                     
\usepackage[normalem]{ulem}                                                     
\usepackage [scaled=.90]{helvet}                     
\usepackage {courier}                                
\usepackage {graphicx}                               
\usepackage {array}                                  
\usepackage {longtable}                              
\usepackage{fancyvrb}
\usepackage{enumerate} 
\usepackage{ifpdf}
\usepackage{caption}
\usepackage{mathrsfs}

\usepackage{setspace}

\usepackage{marvosym}

\usepackage{mathtools}
\mathtoolsset{showonlyrefs} 

\newtheorem{theorem}{Theorem}[section]

\newtheorem{lemma}[theorem]{Lemma}

\newtheorem{proposition}[theorem]{Proposition}

\theoremstyle{definition}
\newtheorem{assumption}{Assumption}[section]
\newtheorem{definition}[theorem]{Definition}
\newtheorem{example}{Example}

\newcommand{\exclude}[1]{}

\newcommand{\1}{{\mathbf{1}}}
\newcommand{\E}{{\mathbb{E}}}

\newcommand{\N}{{\mathbb{N}}}
\renewcommand{\P}{{\mathbb{P}}}

\newcommand{\R}{{\mathbb{R}}}

\newcommand{\sgn}{\operatorname{sgn}}

\definecolor{darkgreen}{rgb}{0,0.5,0}
\definecolor{lightgreen}{rgb}{0.5,0.9,0.5}
\definecolor{magenta}{rgb}{0.75,0,0.25}

\definecolor{violet}{rgb}{0.25,0,0.75}

\newcommand{\appx}{X^{(\delta)}}
\newcommand{\appxx}{X^{(\delta),x}}

\newcommand{\appz}{Z^{(\delta)}}
\newcommand{\eu}{{\underline{u}}}
\newcommand{\es}{{\underline{s}}}
\newcommand{\et}{{\underline{t}}}

\renewcommand{\P}{{\mathbb P}}

\newcommand{\cF}{{\cal F}}

\newcommand{\be}{\begin{equation}}
\newcommand{\ee}{\end{equation}}
\newcommand{\bea}{\begin{eqnarray}}
\newcommand{\eea}{\end{eqnarray}}
\newcommand{\beast}{\begin{eqnarray*}}
\newcommand{\eeast}{\end{eqnarray*}}
\newcommand{\bproof}{\begin{proof}}
\newcommand{\eproof}{\end{proof}}

\title{Existence, uniqueness, and approximation of solutions of jump-diffusion SDEs with discontinuous drift}

\author{Pawe{\l} Przyby{\l}owicz \and Michaela Sz\"olgyenyi}
\date{Preprint, January 2021}

\begin{document}

\maketitle


\begin{abstract}
In this paper we study jump-diffusion stochastic differential equations (SDEs) with a discontinuous drift coefficient and a possibly degenerate diffusion coefficient.
Such SDEs appear in applications such as optimal control problems in energy markets.
We prove existence and uniqueness of strong solutions. In addition we study the strong convergence order of the Euler-Maruyama scheme and recover the optimal rate $1/2$.\\

\end{abstract}

\section{Introduction}\label{sec:intro}

We consider a time-homogeneous jump-diffusion stochastic differential equation (SDE)
\begin{align}\label{eq:SDE}
dX_t &= \mu(X_t)  dt + \sigma (X_t) dW_t+\rho(X_{t-})dN_t , \quad t\in[0,T], \quad X_0=\xi,
 \end{align}
where $\xi\in\R$, $\mu,\sigma,\rho\colon\R\to \R$ are measurable functions, $T\in(0,\infty)$, $W=(W_t)_{t\in[0,T]}$ is a
standard Brownian motion and $(N_t)_{t\in[0,T]}$ is a Poisson process with Borel measurable and bounded intensity $\lambda\colon [0,T]\to (0,\infty)$ on the filtered probability space $(\Omega,\cF,(\cF_t)_{t\in[0,T]},\P)$ that satisfies the usual conditions.

Furthermore, let $N\in\N$, define the equidistant time grid $0=t_0<t_1<\dots<t_N=T$ with $t_{k+1}-t_k=\delta$ for all $k\in\{0,\dots,N-1\}$ and denote for all $t\in[0,T]$, $\et:=\max\{ t_k\colon t\ge t_k \}$.
The time-continuous Euler-Maruyama (EM) scheme is given by $\appx_0=\xi$ and
\begin{align} \label{euler}
\appx_{t} = \appx_{\et}+ \mu(\appx_{\et})(t-\et) + \sigma(\appx_{\et})(W_{t}-W_{\et}) + \rho(\appx_{\et})(N_{t}-N_{\et}), \quad t\in(0,T].
\end{align}

In case the coefficients $\mu$, $\sigma$, and $\rho$ are Lipschitz, it is well known that SDE \eqref{eq:SDE} admits a unique strong solution which can be approximated with the EM scheme at strong convergence order $1/2$.

In this work we allow $\mu$ to be discontinuous in a finite number of points.
This is relevant for example for modelling energy prices, where jumps in the paths are a stylised fact, see, e.g., \cite{benth2014}.
Control actions on energy markets often lead to discontinuities in the drift of the controlled process, cf., e.g., \cite{sz2016a,shardin2017}.

We study existence and uniqueness of solutions to SDE \eqref{eq:SDE} as well as numerical approximations of this solution.\\

In the jump-free case, SDEs with discontinuous drift have been studied intensively in recent years.
For existence and uniqueness results see the classical papers \cite{Zvonkin1974,Veretennikov1981,Veretennikov1982,Veretennikov1984}, as well as newer results, where boundedness of the coefficients and non-degeneracy of the diffusion coefficient is no longer needed, see \cite{sz14,sz2016a,sz15,sz2017a}.
For approximation results see \cite{gyongy1998,halidias2006,halidias2008,sz15,ngo2016,sz2017a,ngo2017a,ngo2017b,MENOUKEUPAMEN}.   
In the scalar case the best known results are $L^p$-order $1/2$ of the EM scheme, see \cite{muellergronbach2019} and $L^p$-order $3/4$ of a transformation-based Milstein-type scheme, see \cite{muellergronbach2019b}. 
In the multidimensional setting the best known results are $L^2$-order $1/4-$ of the EM scheme, see \cite{sz2018a} and $L^2$-order $1/2-$ of an adaptive EM scheme, see \cite{sz2018b}. 
In the special case of additive noise the best known results are $L^2$-order $(1+\kappa)/2-$ assuming piecewise Sobolev-Slobodeckij-type regularity of order $\kappa \in (0,1)$ for the drift, see \cite{NS19} and \cite{gerencser2018}, where they prove $L^2$-order $1/2-$ for the case where the drift is only bounded and $L^1$ (i.e.~the case $\kappa=0$).
A lower error bound of order $1$ for the pointwise $L^1$-error is proven in \cite{hefter2018}. Lower bounds will be also studied in a forthcoming paper by M\"uller-Gronbach and Yaroslavtseva.

In the case of presence of jumps in the driving process, to the best of our knowledge there are no results available for SDEs with discontinuous drift so far. In the case of continuous coefficients however, the number of publications is still growing. This is due to the already mentioned fact that jumps often arise in models for energy markets, financial markets, or physical phenomena, see for example \cite{PBL,situ2006}. The research directions cover for example classical It\^o-Taylor approximations as in \cite{gar1,gar2,PBL},  construction of Runge-Kutta methods as in \cite{br,PBL}, approximation of jump-diffusion SDEs under nonstandard assumptions as in \cite{kds,dflm,dffm,hk1,hk2,hk3}, multilevel Monte Carlo methods for weak approximation as in \cite{dh}, and asymptotically optimal approximations of solutions of such SDEs as in \cite{akpp,pp1,pp2,pp3}.\\

The current paper consists of two main contributions:~the first existence and uniqueness result for jump-diffusion SDEs with discontinuous drift and consequently the first approximation result for solutions to such SDEs. We obtain the optimal $L^2$-order $1/2$ for the EM scheme.

We reach our goal by adopting ideas from the jump-free case from \cite{sz15,sz2018a,muellergronbach2019} and extending everything to the case of presence of jumps.
An interesting by-product is an occupation time result for the EM process.

\section{Preliminaries}\label{sec:pre}

In the following we denote by $L_f$ the Lipschitz constant of a generic function $f$,
we define $\kappa\colon  [0,T]\to (0,\infty)$ with $\kappa_t=\int_0^{t}\lambda (s)ds$, and we denote by $\tilde N=(\tilde N_t)_{t\in[0,T]}$ the compensated Poisson process, that is  $\tilde N_t=N_t-\kappa_t$ for all $t\in[0,T]$. Note that $\tilde N$ is a square integrable $(\cF_t)_{t\in[0,T]}$-martingale.

In order to define assumptions on the drift coefficient, we recall the following definition.

\begin{definition}[\text{\cite[Definition 2.1]{sz2017a}}]\label{def:pwlip-1dim}
Let $I\subseteq\R$ be an interval and let $m\in\N$. We say a function $f\colon I\to\R$
is piecewise Lipschitz, if there are finitely many points $\zeta_1<\ldots<\zeta_m\in
I$ such that $f$ is Lipschitz on each of the intervals $(-\infty,\zeta_1)\cap I,
(\zeta_m,\infty)\cap I$, and $(\zeta_k,\zeta_{k+1}),k=1,\ldots,m-1$.
\end{definition}

\begin{assumption}\label{ass:ex-un}We assume the following on the coefficients of \eqref{eq:SDE}:
\begin{itemize}
\item[(ass-$\mu$)] The drift coefficient $\mu\colon\R \to \R$ is piecewise Lipschitz with $m\in\N$ discontinuities in the points $\zeta_1,\dots,\zeta_m\in\R$.
\item[(ass-$\sigma$)] The diffusion coefficient $\sigma\colon\R\to\R$ is Lipschitz and for all $k\in\{1,\dots,m\}$, $\sigma(\zeta_k) \neq 0$.
\item[(ass-$\rho$)] The jump coefficient $\rho\colon\R\to\R$ is Lipschitz.
\end{itemize}
\end{assumption}

\begin{lemma}
\label{lin_growth_c}
 Let Assumptions \ref{ass:ex-un} hold. Then $\mu$, $\sigma$, and $\rho$ satisfy a linear growth condition, that is
 there exist constants $c_\mu, c_\sigma, c_\rho \in (0,\infty)$ such that
 \begin{equation*}
 |\mu(x)|\le c_\mu(1+ |x|), \qquad |\sigma(x)|\le c_\sigma(1+ |x|), \qquad |\rho(x)|\le c_\rho(1+ |x|).
 \end{equation*}
\end{lemma}
\begin{proof}
We have that $|\sigma(x)|\le |\sigma(x)-\sigma(0)|+|\sigma(0)|\le L_\sigma |x|+|\sigma(0)|$. Setting $c_\sigma=\max\{L_\sigma,|\sigma(0)|\}$ we get $|\sigma(x)|\le c_\sigma(1+ |x|)$.
The analog estimate holds for $\rho$. For $x\in(-\infty,\zeta_1)$ we have that there exists an $\varepsilon \in (0,\infty)$ with $\zeta_1-\varepsilon>x$. With this $|\mu(x)|\le |\mu(x)-\mu(\zeta_1-\varepsilon)|+|\mu(\zeta_1-\varepsilon)|\le L_\mu |x-(\zeta_1-\varepsilon)|+|\mu(\zeta_1-\varepsilon)|\le L_\mu |x|+L_\mu |\zeta_1-\varepsilon|+|\mu(\zeta_1-\varepsilon)|$. Setting $c_\mu^1=\max\{L_\mu,L_\mu (\zeta_1-\varepsilon)+|\mu(\zeta_1-\varepsilon)|\}$ we get $|\mu(x)|\le c_\mu^1(1+ |x|)$. In the same way for $x\in(\zeta_m,\infty)$ there exists $c_\mu^2\in(0,\infty)$ with $|\mu(x)|\le c_\mu^2(1+ |x|)$. In the compact interval $[\zeta_1,\zeta_m]$, $\mu$ is bounded by a constant $c_\mu^3\in(0,\infty)$. Setting $c_\mu=\max\{c_\mu^1,c_\mu^2,c_\mu^3\}$ proves the lemma.
\end{proof}

\subsection*{The transform}
We will apply a transform $G\colon \R\to \R$ from \cite{sz2017a} that has the property that
the process formally defined by $Z=G(X)$ satisfies an SDE with 
Lipschitz coefficients and therefore has a solution by classical results, see \cite[p.~255, Theorem 6]{protter2004}.

The function $G$ is chosen so that it impacts the coefficients of the SDE \eqref{eq:SDE} only locally around the points of discontinuity of the drift.
This behaviour is ensured by incorporating a bump function $\phi\colon \R \to \R$ into $G$, which is defined by 
\begin{align}\label{eq:bump}
\phi(u)=
\begin{cases}
(1+u)^3(1-u)^3 & \text{if } |u|\le 1,\\
0 & \text{otherwise}.
\end{cases}
\end{align}
With this the transform $G$ is defined by
\begin{align}\label{eq:G-1d}
G(x)=x+ \sum_{k=1}^m \alpha_k\phi\!\left(\frac{x-\zeta_k}{c}\right)(x-\zeta_k)|x-\zeta_k| ,
\end{align}
with
\begin{align*}
& \R\setminus\{0\} \ni \alpha_k
=\frac{\mu(\zeta_k-)-\mu(\zeta_k+)}{2\sigma(\zeta_k)^2},
\qquad k\in\{1,\dots,m\},
\\ 
& (0,\infty)\ni c<\min\bigg\{\min\limits_{1\le k\le m}\frac{1}{6|\alpha_k|},\min\limits_{1\le k\le m-1}\frac{\zeta_{k+1}-\zeta_k}{2}\bigg\}.
\end{align*}
Note that $c$ is chosen such that for all $x\in\R$, $G'(x)>0$, so that $G$ has a global inverse $G^{-1}\colon \R\to\R$.
The transformation $G$ and its inverse $G^{-1}$ are Lipschitz and the function $G\in C^1_b$, that is it is continuously differentiable with bounded derivative.
Furthermore, $G'$ is piecewise Lipschitz, since it is differentiable on $\R\setminus \{\zeta_1,\dots,\zeta_m\}$ with bounded derivative, see \cite[Lemma 3.8]{sz2017a}.
Hence, $G'$ is Lipschitz, since it is piecewise Lipschitz and continuous, see \cite[Lemma 2.2]{sz15}.
These properties are proven in \cite{sz2017a}.

Finally, note that the choice of $G$ is not unique. In fact, it suffices to ensure existence of a function $G$ satisfying these properties.

\section{Existence and uniqueness result}

We are going to prove our first main result.

\begin{theorem}
\label{exun}
Let Assumption \ref{ass:ex-un} hold.
Then the SDE \eqref{eq:SDE} has a unique global strong solution.
\end{theorem}

\begin{proof}
Since $G'$ is Lipschitz, we may apply the Meyer-It\^o formula, which follows from \cite[p.~221, Theorem 71]{protter2004}, to $Z=G(X)$ and get
\begin{equation}
\label{eq:Z}
	dZ_t=\tilde\mu (Z_t)dt+\tilde\sigma(Z_t)dW_t+\tilde \rho(Z_{t-})dN_t,
\end{equation}
where for all $z\in\R$,
\begin{align}
\tilde\mu(z)&=G'(G^{-1}(z)) \mu(G^{-1}(z))+\frac{1}{2}G''(G^{-1}(z))\sigma(G^{-1}(z))^2,\\
\tilde\sigma(z)&=G'(G^{-1}(z))\sigma(G^{-1}(z)),\\
\label{t_rho_def}
\tilde \rho(z)&=G(G^{-1}(z)+\rho(G^{-1}(z)))-G(G^{-1}(z))=G(G^{-1}(z)+\rho(G^{-1}(z)))-z.
\end{align}

In \cite{sz2017a} it is shown that $\tilde\mu$ and $\tilde\sigma$ are Lipschitz.
The jump coefficient $\tilde \rho$ is Lipschitz due to the global Lipschitz continuity of $G$ and $G^{-1}$. 
Hence, the SDE for $Z$, that is \eqref{eq:Z} with initial condition $Z(0)=G(\xi)$, has a unique global strong solution by \cite[p.~255, Theorem 6]{protter2004}.

Now observe that $(G^{-1})'(z)=1/G'(G^{-1}(z))$ is absolutely continuous since it is Lipschitz.
Moreover, $G^{-1}(Z_{t-})=\lim\limits_{s\to t-}G^{-1}(Z_s)=X_{t-}$, and by \eqref{t_rho_def} we have
\begin{equation}
	Z_{t-}+\tilde\rho(Z_{t-})=G(X_{t-}+\rho(X_{t-})).
\end{equation}
This implies that
\begin{equation}
	G^{-1}(Z_{t-}+\tilde \rho (Z_{t-}))-G^{-1}(Z_{t-})=\rho(X_{t-}).
\end{equation}
Therefore, again by using Meyer-It\^o formula
\begin{align*}
	dG^{-1}(Z_t)&=((G^{-1})'(Z_t)\tilde\mu (Z_t)+\frac{1}{2}(G^{-1})''(Z_t)\tilde\sigma(Z_t)^2)dt+(G^{-1})'(Z_t)\tilde\sigma(Z_t)dW_t
	\\&\quad
	+(G^{-1}(Z_{t-}+\tilde \rho (Z_{t-}))-G^{-1}(Z_{t-}))dN_t
	\\&=\mu(X_t)dt+\sigma(X_t)dW_t+\rho (X_{t-})dN_t.
\end{align*}
\end{proof}

\section{Convergence of the Euler-Maruyama method}
\label{sec:EM}

Our convergence proof is based on a transformation trick from \cite{sz15,sz2018a} in combination with ideas from \cite{muellergronbach2019} for the estimation of discontinuity crossing probabilities. By extending both to the case of presence of jumps and proving an occupation time result for the EM process, this leads to the optimal convergence order $1/2$.

\subsection{Preparatory lemmas}

In this section we present several lemmas, cf.~the results for the jump-free case in \cite{muellergronbach2019}, which we need for the proof of the main result of Section \ref{sec:EM}.

\begin{lemma} \label{msq-appx} Let Assumptions \ref{ass:ex-un} hold and let $p\in[2,\infty)$. There exist a constant $C_p^{(\text{M})}\in(0,\infty)$   such that for all $\delta\in(0,1)$,
$$ \E \Big[ \sup_{t \in [0,T]} |\appx_t  |^p   \Big] \leq  C_p^{(\text{M})}    $$  and such that for all $s,t \in[0,T]$, $s<t$,
$$ \E [ |\appx_t - \appx_s |^p ]  \leq C_p^{(\text{M})} \cdot |t-s|.$$
\end{lemma}
\begin{proof}
There exists a constant $c_1\in(0,\infty)$ such that for all $k \in\{0,1,\ldots N-1\}$,
\begin{equation}
\begin{aligned}
  | \appx_{t_{k+1}} |^p & \leq c_1\Bigl( |  \appx_{t_{k}}|^p + | \mu( \appx_{t_{k}})|^p\cdot  \delta^p
  +   | \sigma(\appx_{t_k}) |^p \cdot | W_{t_{k+1}}-W_{t_k}  |^p\notag\\
&\quad +| \rho(\appx_{t_k}) |^p \cdot | N_{t_{k+1}}-N_{t_k}  |^p\Bigr).
\end{aligned}
 \end{equation} 
Furthermore there exists a constant $c_2\in(0,\infty)$ such that
\begin{equation}
	\E[|W_{t_{k+1}}-W_{t_k}|^p] \leq c_2 \delta^{p/2}
\end{equation}
and by \cite[Inequality (3.20)]{dflm},
\begin{equation}\label{est-N}
	\E[|N_{t_{k+1}}-N_{t_k}|^p] \leq c_2 \delta.
\end{equation}
This and the linear growth of $\mu$, $\sigma$, and $\rho$ imply the existence of a constant $c_3\in(0,\infty)$ such 
that $$  \E \left[|\appx_{t_{k+1}} |^p\right] \leq c_3 \left(1+ \E \left[|  \appx_{t_{k}}|^p\right]\right). $$
Since $\E[|X_0|^p]<\infty$ it follows that
\begin{align}   \label{finite_m_1} \max\limits_{k\in\{0,1,\ldots,N\}}\E \left[| \appx_{t_{k}} |^p \right]< \infty.   \end{align}
We also have for $s \in [0,T]$ that there exists a constant $c_4\in(0,\infty)$ such that
\begin{equation}
\begin{aligned}
  \sup_{t \in [0,s]} |\appx_t|^p& \leq c_4\bigg( |X_0|^p+ \sup_{t \in [0,s]} \Big | \int_0^t  \mu(\appx_{\eu}) du  \Big|^p
 \\&\quad+
   \sup_{t \in [0,s]} \Big | \int_0^t  \sigma(\appx_{\eu}) dW_u  \Big|^p+  \sup_{t \in [0,s]} \Big | \int_0^t  \rho(\appx_{\eu}) dN_u  \Big|^p\bigg).
\end{aligned}
\end{equation}
The Burkholder-Davis-Gundy inequality, Doob's maximal inequality for c\'adl\'ag martingales, \cite[Lemma 2.1]{maghsoodi1996}, and the linear  growth condition on $\sigma$, $\rho$ ensure the existence of constants $c_5,c_6\in(0,\infty)$ such that
\begin{equation}
	\E\bigg[\sup_{t \in [0,s]} \Big | \int_0^t  \sigma(\appx_{\eu}) dW_u  \Big|^p\bigg]\leq c_5\, \E\!\left[\int_0^s |\sigma(\appx_{\eu})|^pdu\right]\leq c_6\left(1+\int_0^s\E\!\left[|\appx_{\eu}|^p\right] du\right)
\end{equation}
and since $N_u=\tilde N_u+\kappa_u$,
\begin{equation}
	\E\bigg[\sup_{t \in [0,s]} \Big | \int_0^t  \rho(\appx_{\eu}) dN_u  \Big|^p\bigg]\leq c_5 \E\!\left[\int_0^s |\rho(\appx_{\eu})|^p du\right]\leq c_6\left(1+\int_0^s\E\!\left[|\appx_{\eu}|^p\right] du\right).
\end{equation}
Since $\mu$ is of at most linear growth, an analogous estimate for the Lebesgue integral holds.
Hence,
\begin{align} \label{eq_moment_2} \E \bigg[ \sup_{t \in [0,s]} |\appx_t|^p \bigg] \leq c_2 \left(1 +  \int_0^s  \E\! \left[| \appx_{\eu} |^p \right] du \right), \quad s \in [0,T]. \end{align}
By  \eqref{finite_m_1} we now obtain that
\begin{align}   \label{finite_m_2}  \E \bigg[ \sup_{t \in [0,T]} |\appx_t|^p \l \bigg] < \infty .   \end{align}
Equation \eqref{eq_moment_2} also yields
$$  \E  \bigg[ \sup_{t \in [0,s]} |\appx_t  |^p  \bigg]  \leq c_2  + c_2 \int_0^s   \E  \bigg[ \sup_{t \in [0,u]} |\appx_t  |^p \bigg] du.$$
Since  \eqref{finite_m_2} holds and the function $[0,T]\ni t\mapsto \E  \Big[ \sup_{s \in [0,t]} |\appx_s  |^p\Bigr]$ is Borel measurable (as a nondecreasing mapping),   Gronwall's Lemma yields the first assertion.

For the second statement note that for all $s,t\in [0,T]$, $s<t$, there exists a constant $c_7\in(0\infty)$ so that it holds
\begin{equation}
\begin{aligned}
\E\big[ | \appx_t - \appx_s|^{p}\big]& \leq  c_7\bigg( \E\bigg[\Big | \int_s^t \mu(\appx_{\eu}) du \Big |^p\bigg]+\E\bigg[\Big | \int_s^t \sigma(\appx_{\eu}) dW_u \Big |^p\bigg]+\E\bigg[\Big | \int_s^t \rho(\appx_{\eu}) dN_u \Big |^p\bigg]\bigg).
\end{aligned}
\end{equation}
The H\"older inequality, the Burkholder-Davis-Gundy inequality, \cite[Lemma 2.1]{maghsoodi1996}, and the linear growth condition of the coefficients together with the first assertion yield the statement.
\end{proof}

Note that we consider the $L^2$-error and not the $L^p$-error as in the jump-free case in \cite{muellergronbach2019}, since due to \eqref{est-N} we will not get a better estimate for $p>2$ anyhow, cf.~\cite[Remark 3.14]{dflm}.
For later use we define $C^{(\text{M})} = \max\{C_p^{(\text{M})}\colon p\in\{1,\dots,8\}\}$.

\subsubsection{Estimation of the occupation time of the Euler-Maruyama process}
In this subsection we need to make the dependence on the initial value explicit in the notation. For all $x\in\mathbb{R}$ denote by $X^x$ the unique strong solution of \eqref{eq:SDE} with initial condition $X^x_0=x$ and by $X^{(\delta),x}$ the solution of the time-continuous version of the Euler-Maruyama scheme \eqref{euler} starting at $\appxx_0=x$.
Note that from the proof of Lemma \ref{msq-appx} it follows that there exists $ C^{(\text{I})}\in (0,\infty)$ such that for all $x\in\mathbb{R}$, $s,t\in [0,T]$, $s<t$, $\delta \in(0,1)$,
\begin{equation} \label{mom_est_euler_x}
	\Big(\E\Big[\sup\limits_{0\leq t\leq T}|\appxx_t|^2\Big]\Big)^{\!1/2}\leq C^{(\text{I})}(1+|x|),
\end{equation}
\begin{equation}  \label{mean_sqrt_reg_x}
	\big(\E\big[|\appxx_{t}-\appxx_{s}|^2\big]\big)^{\!1/2}\leq C^{(\text{I})}(1+|x|)|t-s|^{1/2}.
\end{equation}
\begin{lemma}\label{occup_time}
	Let Assumptions \ref{ass:ex-un} hold. Then there exists $C^{(\text{O})}\in (0,\infty)$ such that for all $k\in\{1,\dots,m\}$, $x\in\mathbb{R}$, 
$\delta\in(0,1)$, $\varepsilon \in (0,\infty)$ it holds
	\begin{equation}
		\int\limits_0^T\mathbb{P}(|X_t^{(\delta),x}-\zeta_k|\leq\varepsilon)dt\leq C^{(\text{O})}(1+x^2)(\varepsilon+\delta^{1/2}).
	\end{equation}
\end{lemma}
\begin{proof} By \cite[Lemma 158]{situ2006} we have for all $a\in \R$ that 
\begin{equation*}
\begin{aligned}
	&L_t^a(\appxx)=|\appxx_t-a|-|x-a|-\int\limits_0^t \sgn(\appxx_{s-}-a)d\appxx_s\\
	&\qquad-\int\limits_0^t\Bigl(|\appxx_{s-}+\rho(\appxx_{\es})-a|-|\appxx_{s-}-a|-\sgn(\appxx_{s-}-a)\rho(\appxx_{\es})\Bigr)dN_s,
	\end{aligned}
\end{equation*}
where $L_t^a(\appxx)$ is the local time of the semi-martingale $\appxx$ in $a$. Since $L_t^a(\appxx)=|L_t^a(\appxx)|\geq 0$, we get
\begin{equation}\label{helpeqn3}
\begin{aligned}
	&L_t^a(\appxx)\leq |\appxx_t-x|+\Bigl|\int\limits_0^t \sgn(\appxx_{s-}-a)d\appxx_s\Bigl|\\
	&\qquad+\Bigl|\int\limits_0^t\Bigl(|\appxx_{s-}+\rho(\appxx_{\es})-a|-|\appxx_{s-}-a|-\sgn(\appxx_{s-}-a)\rho(\appxx_{\es})\Bigr)dN_s\Bigl|.
	\end{aligned}
\end{equation}
By Lemma \ref{lin_growth_c} there exists $c_1\in(0,\infty)$ such that
\begin{equation}\label{helpeqn2}
\begin{aligned}
	&\Bigl|\int\limits_0^t \sgn(\appxx_{s-}-a)d\appxx_s\Bigl|\leq \int\limits_0^t\Bigl(|\mu(\appxx_{\es})|+\|\lambda\|_{\infty}|\rho(\appxx_{\es})|\Bigr)ds\\
	&\qquad+\Bigl|\int\limits_0^t \sgn(\appxx_{s-}-a)\sigma(\appxx_{\es})dW_s\Bigl|
	+\Bigl|\int\limits_0^t \sgn(\appxx_{s-}-a)\rho(\appxx_{\es})d\tilde N_s\Bigl|\\
	&\quad\leq c_1(1+\sup\limits_{t\in[0,T]}|\appxx_t|)+\Bigl|\int\limits_0^t \sgn(\appxx_{s-}-a)\sigma(\appxx_{\es})dW_s\Bigl|\\
	&\qquad+\Bigl|\int\limits_0^t \sgn(\appxx_{s-}-a)\rho(\appxx_{\es})d\tilde N_s\Bigl|
	\end{aligned}
\end{equation}
and
\begin{equation}\label{helpeqn1}
 \begin{aligned}
&\Bigl|\int\limits_0^t\Bigl(|\appxx_{s-}+\rho(\appxx_{\es})-a|-|\appxx_{s-}-a|-\sgn(\appxx_{s-}-a)\rho(\appxx_{\es})\Bigr)dN_s\Bigl|\\
&\quad\leq\Bigl|\int\limits_0^t\Bigl(|\appxx_{s-}+\rho(\appxx_{\es})-a|-|\appxx_{s-}-a|-\sgn(\appxx_{s-}-a)\rho(\appxx_{\es})\Bigr)d\tilde N_s\Bigl|\\
&\qquad+\Bigl|\int\limits_0^t\Bigl(|\appxx_{s-}+\rho(\appxx_{\es})-a|-|\appxx_{s-}-a|-\sgn(\appxx_{s-}-a)\rho(\appxx_{\es})\Bigr)\lambda(s)ds \Bigl|\\
&\quad\leq\Bigl|\int\limits_0^t\Bigl(|\appxx_{s-}+\rho(\appxx_{\es})-a|-|\appxx_{s-}-a|-\sgn(\appxx_{s-}-a)\rho(\appxx_{\es})\Bigr)d\tilde N_s\Bigl|\\
&\qquad+2\|\lambda\|_{\infty}\int\limits_0^t|\rho(\appxx_{\es})|ds\\
&\quad\leq \Bigl|\int\limits_0^t\Bigl(|\appxx_{s-}+\rho(\appxx_{\es})-a|-|\appxx_{s-}-a|-\sgn(\appxx_{s-}-a)\rho(\appxx_{\es})\Bigr)d\tilde N_s\Bigl|\\
&\qquad+c_1\Bigl(1+\sup\limits_{t\in[0,T]}|\appxx_t|\Bigr). 
	\end{aligned}
\end{equation}
By \eqref{mom_est_euler_x} there exists $c_2\in(0,\infty)$ such that
\begin{equation*}
\begin{aligned}
	\E\bigg[\Bigl|\int\limits_0^t \sgn(\appxx_{s-}-a)\sigma(\appxx_{\es})dW_s\Bigl|^2\bigg]&\leq \E\bigg[\int\limits_0^t|\sigma(\appxx_{\es})|^2ds\bigg]\leq c_2(1+|x|)^2,\\
	\E\bigg[\Bigl|\int\limits_0^t \sgn(\appxx_{s-}-a)\rho(\appxx_{\es})d\tilde N_s\Bigl|^2\bigg]&\leq \|\lambda\|_{\infty}\E\bigg[\int\limits_0^t|\rho(\appxx_{\es})|^2ds\bigg]\leq c_2(1+|x|)^2,
	\end{aligned}
\end{equation*}
and
\begin{equation*}
\begin{aligned}
&	\E\bigg[\Bigl|\int\limits_0^t\Bigl(|\appxx_{s-}+\rho(\appxx_{\es})-a|-|\appxx_{s-}-a|-\sgn(\appxx_{s-}-a)\rho(\appxx_{\es})\Bigr)d\tilde N_s\Bigl|^2\bigg]\\
&\quad=\E\bigg[\int\limits_0^t\Bigl(|\appxx_{s-}+\rho(\appxx_{\es})-a|-|\appxx_{s-}-a|-\sgn(\appxx_{s-}-a)\rho(\appxx_{\es})\Bigr)^2\lambda(s)ds\bigg]\\
&\quad\leq 2\|\lambda\|_{\infty}\E\bigg[\int\limits_0^t \Bigl||\appxx_{s-}+\rho(\appxx_{\es})-a|-|\appxx_{s-}-a|\Bigl|^2ds\bigg]+2\|\lambda\|_{\infty}\E\bigg[\int\limits_0^t|\rho(\appxx_{\es})|^2ds\bigg]\\
&\quad\leq 4\|\lambda\|_{\infty}\E\bigg[\int\limits_0^t|\rho(\appxx_{\es})|^2ds\bigg]
\leq
c_2(1+|x|)^2.
	\end{aligned}
\end{equation*}
Together with \eqref{helpeqn2} respectively \eqref{helpeqn1} this gives that there exist constants $c_3,c_4\in(0,\infty)$ such that
\begin{equation*}
\begin{aligned}
	&\E\bigg[\Bigl|\int\limits_0^t \sgn(\appxx_{s-}-a)d\appxx_s\Bigl|\bigg]\leq c_3(1+|x|)+\Biggl(\E\bigg[\Bigl|\int\limits_0^t \sgn(\appxx_{s-}-a)\sigma(\appxx_{\es})dW_s\Bigr|^2\bigg]\Biggr)^{\!1/2}\\
	&\qquad+\Biggl(\E\bigg[\Bigl|\int\limits_0^t \sgn(\appxx_{s-}-a)\rho(\appxx_{\es})d\tilde N_s\Bigl|^2\bigg]\Biggr)^{\!1/2}\leq c_4(1+|x|),
	\end{aligned}
\end{equation*}
respectively
\begin{equation*}
\begin{aligned}
&\E\bigg[\Bigl|\int\limits_0^t\Bigl(|\appxx_{s-}+\rho(\appxx_{\es})-a|-|\appxx_{s-}-a|-\sgn(\appxx_{s-}-a)\rho(\appxx_{\es})\Bigr)dN_s\Bigl|\bigg]\\
&\quad\leq\Bigg(\E\bigg[\Bigl|\int\limits_0^t\Bigl(|\appxx_{s-}+\rho(\appxx_{\es})-a|-|\appxx_{s-}-a|-\sgn(\appxx_{s-}-a)\rho(\appxx_{\es})\Bigr)d\tilde N_s\Bigl|^2\bigg]\Bigg)^{1/2}\\
&\qquad+	c_3\Bigl(1+\E\Big[\sup\limits_{t\in[0,T]}|\appxx_t|\Big]\Bigr)
\le c_4(1+|x|).
\end{aligned}
\end{equation*}
Combining these estimates with \eqref{helpeqn3} shows
\begin{equation}
	\E[L_t^a(\appxx)]\leq c_4(1+|x|).
\end{equation}
Note that the continuous martingale part of the semi-martingale \eqref{euler} starting at $\appxx_0=x$ is given by
\begin{equation}
	M_t=\int\limits_0^t\sigma(\appxx_{\es})dW_s, \quad t\in [0,T].
\end{equation}
and its predictable quadratic variation is
\begin{equation}
	\langle M\rangle_t=\int\limits_0^t|\sigma(\appxx_{\es})|^2ds, \quad t\in [0,T].
\end{equation}
Therefore, by \cite[Lemma 159]{situ2006}, we have for all $\varepsilon\in (0,\infty)$, $t\in [0,T]$ that
\begin{equation}\label{helpeqn4}
\begin{aligned}
&\E\bigg[\int\limits_0^t \mathbf{1}_{[\zeta_k -\varepsilon,\zeta_k +\varepsilon]}(\appxx_s)\cdot|\sigma(\appxx_{\es})|^2ds\bigg]=\E\bigg[\int\limits_0^t \mathbf{1}_{[\zeta_k -\varepsilon,\zeta_k +\varepsilon]}(\appxx_s)d\langle M\rangle_s\bigg]\\
&\quad=\int\limits_{\mathbb{R}}\mathbf{1}_{[\zeta_k -\varepsilon,\zeta_k +\varepsilon]}(a)\cdot\E[L_t^a(\appxx)]da\leq 2c_4(1+|x|)\cdot\varepsilon.
\end{aligned}
\end{equation}
Since $\sigma$ is Lipschitz and of at most linear growth, 
\begin{equation*}
\begin{aligned}
&|\sigma^2(\appxx_s)-\sigma^2(\appxx_{\es})|\leq |\sigma(\appxx_s)-\sigma(\appxx_{\es})|\cdot (|\sigma(\appxx_s)|+|\sigma(\appxx_{\es})|)\notag\\
&\quad\leq 2 c_\sigma|\appxx_s-\appxx_{\es}|\cdot \Big(1+\sup\limits_{t\in[0,T]}|\appxx_t|\Big).
\end{aligned}
\end{equation*}
Hence, by the Cauchy-Schwarz inequality, \eqref{mom_est_euler_x}, and \eqref{mean_sqrt_reg_x} there exists $c_5\in(0,\infty)$ such that
\begin{equation*}
\begin{aligned}
&\E\!\left[|\sigma^2(\appxx_s)-\sigma^2(\appxx_{\es})|\right]
\leq 2c_\sigma\left(\E[\ |\appxx_s-\appxx_{\es}|^2]\right)^{1/2}\cdot \Bigl(\E\!\ \Bigl(1+\sup\limits_{0\leq t \leq T}|\appxx_t|\Bigr)^2\Bigr)^{1/2}\\
&\quad\leq 2c_\sigma(\E\!\ |\appxx_s-\appxx_{\es}|^2)^{1/2}\cdot \Big(1+\Bigl(\E\Big[\sup\limits_{t\in[0,T]}|\appxx_t|^2\Big]\Bigr)^{\!1/2}\Bigr)
\leq 2c_5 (1+|x|)^2\cdot |s-{\es}|^{1/2}.
\end{aligned}
\end{equation*}
Thus we have for all $t\in [0,T]$,
\begin{equation}\label{helpeqn5}
\E\bigg[\int\limits_0^t|\sigma^2(\appxx_s)-\sigma^2(\appxx_{\es})|ds\bigg]\leq c_5(1+|x|)^2\sum\limits_{k=0}^{N-1}\int\limits_{t_k}^{t_{k+1}}(s-t_k)^{1/2}ds
\leq c_5 T (1+x^2)\delta^{1/2}.
\end{equation}
From the continuity of $\sigma$ and by the assumption that $\sigma(\zeta_k )\neq 0$, we get that there exist $c_6,\varepsilon_0\in (0,\infty)$ such that
\begin{equation}
	\inf\limits_{z\in (\zeta_k -\varepsilon_0,\zeta_k +\varepsilon_0)}\sigma^2(z)\geq c_6.
\end{equation}
Combining this with \eqref{helpeqn4} and \eqref{helpeqn5} we get that there exists $c_7\in(0,\infty)$ such that for all $\varepsilon\in (0,\varepsilon_0)$,
\begin{equation*}
\begin{aligned}
&\int\limits_0^T \mathbb{P}\Bigl( |\appxx_t-\zeta_k |\leq \varepsilon\Bigr)dt
=\frac{1}{c_6}\E\bigg[\int\limits_0^Tc_6 \mathbf{1}_{[\zeta_k -\varepsilon,\zeta_k +\varepsilon]}(\appxx_t)dt\bigg]\\
&\quad\leq \frac{1}{c_6}\E\bigg[\int\limits_0^T \mathbf{1}_{[\zeta_k -\varepsilon,\zeta_k +\varepsilon]}(\appxx_t)\sigma^2(\appxx_t)dt\bigg]\\
&\quad= \frac{1}{c_6}\E\bigg[\int\limits_0^T \mathbf{1}_{[\zeta_k -\varepsilon,\zeta_k +\varepsilon]}(\appxx_t)\sigma^2(\appxx_{\et})dt\bigg]\\
&\qquad+\frac{1}{c_6}\E\bigg[\int\limits_0^T \mathbf{1}_{[\zeta_k -\varepsilon,\zeta_k +\varepsilon]}(\appxx_t)(\sigma^2(\appxx_t)-\sigma^2(\appxx_{\et}))dt\bigg]\\
&\quad\leq \frac{1}{c_6}\E\bigg[\int\limits_0^T \mathbf{1}_{[\zeta_k -\varepsilon,\zeta_k +\varepsilon]}(\appxx_t)\sigma^2(\appxx_{\et})dt\bigg]\\
&\qquad+\frac{1}{c_6}\E\bigg[\int\limits_0^T |\sigma^2(\appxx_t)-\sigma^2(\appxx_{\et})|dt\bigg]\\
&\quad\leq 2c_4(1+|x|)\varepsilon+2c_5 T (1+x^2)\delta^{1/2}\leq c_7(1+x^2)(\varepsilon+\delta^{1/2}).
\end{aligned}
\end{equation*}
For $\varepsilon\in [\varepsilon_0,\infty)$ it trivially holds that
\begin{equation*}
	\int\limits_0^T\mathbb{P}(|X_t^{(\delta),x}-\zeta_k |\leq\varepsilon)dt\leq\ T=\frac{T}{\varepsilon_0} \cdot\varepsilon_0 \le \frac{T}{\varepsilon_0}(1+x^2)(\varepsilon+\delta^{1/2}).
\end{equation*}
Choosing $C^{(\text{O})}=\max\{c_7,\frac{T}{\varepsilon_0}\}$ closes the proof.

\end{proof}
\subsubsection{Estimation of the discontinuity crossing probability}

Note that as in \cite{kds} from now on we write 
$X_{t}$ instead of $X_{t-}$. This is vindicated by
the continuity of the compensators of $W$ and $N$.

Let for all $k\in\{1,\dots,m\}$, $t\in[0,T]$,
$
\mathcal{Z}_k^t=\{\omega\in\Omega\colon(\appx_\et(\omega)-\zeta_k)(\appx_t(\omega)-\zeta_k)\le0\}.
$

\begin{lemma}\label{help0-cross}
Let Assumptions \ref{ass:ex-un} hold. Let $s,t\in[0,T]$ with $\et-s\ge \delta$.
There exists a constant ${C_1}\in(0,\infty)$ such that for all $k\in\{1,\dots,m\}$, $\delta\in(0,1)$ sufficiently small,
\begin{equation*}
 \P(\mathcal{Z}_k^s\cap \mathcal{Z}_k^t) 
\le
C_1\P(\mathcal{Z}_k^s)\delta + {C_1} \cdot \int_\R \P\!\left(\mathcal{Z}_k^s\cap  \left\{|\appx_{\et-(t-\et)}-\zeta_k| \le C_1 \delta^{1/2}(1+|z|)\right\}\right) \cdot e^{-\frac{z^2}{2}} dz.
\end{equation*}
\end{lemma}

\begin{proof}
For treating the Gaussian part we adopt arguments from \cite[Proof of Lemma 5]{muellergronbach2019}.

If $t=\et$, then for all $c_1 \in (0,\infty)$, $z\in\R$ it holds that
\begin{equation*}
\mathcal{Z}_k^t=\{\appx_\et-\zeta_k=0\} \subseteq \left\{|\appx_{\et-(t-\et)}-\zeta_k| \le c_1 \delta^{1/2} (1+|z|)\right\}.
\end{equation*}
So in this case, the assertion of the lemma holds for all ${C_1}\ge 1/\sqrt{2\pi}$.

Now let $t>\et$ and let
\begin{equation*}
\bar W_1=\frac{W_t-W_\et}{\sqrt{t-\et}}, \qquad \bar W_2=\frac{W_\et-W_{\et-(t-\et)}}{\sqrt{t-\et}}, \qquad \bar W_3=\frac{W_{\et-(t-\et)}-W_{\et-\delta}}{\sqrt{\delta-(t-\et)}}, \qquad \bar P=N_t-N_{\et-\delta}.
\end{equation*}
Observe that $\bar W_1, \bar W_2, \bar W_3$ are standard normally distributed, $\bar P$ is Poisson distributed with parameter $\int_{\et-\delta}^t \lambda_s ds$, $\bar W_1, \bar W_2, \bar W_3, \bar P$ are independent, $(\bar W_1, \bar W_2, \bar W_3, \bar P)$ is independent of $\cF_s$ since $s\le \et-\delta$, and $(\bar W_1, \bar W_2)$ is independent of $\cF_{\et-(t-\et)}$.

Let $c_2=\max\{c_\mu,c_\sigma,c_\rho\}$ and let $\delta$ be sufficiently small such that
\begin{equation}\label{star1}
12c_2(1+|\zeta_k|) \cdot \frac{1+\sqrt{2\log(T/\delta)}}{\sqrt{T/\delta}} \le \frac{1}{2}.
\end{equation}
Then note that the following inclusion similar to \cite[(20)]{muellergronbach2019} holds:
\begin{equation}\label{star2}
\begin{aligned}
&\mathcal{Z}_k^t  \cap \{\bar P=0\}\cap \bigg\{\max_{i\in\{1,2,3\}} |\bar W_i|   \le \sqrt{2\log(T/\delta)}\bigg\}\\
&\quad  \subseteq
\{\bar P=0\}\cap\bigg\{|\appx_{\et-(t-\et)}-\zeta_k| \le \frac{12c_2(1+|\zeta_k|)\cdot(1+|\bar W_1|+|\bar W_2|)}{\sqrt{T/\delta}}\bigg\}\\
&\quad  \subseteq
\bigg\{|\appx_{\et-(t-\et)}-\zeta_k| \le \frac{12c_2(1+|\zeta_k|)\cdot(1+|\bar W_1|+|\bar W_2|)}{\sqrt{T/\delta}}\bigg\}.
\end{aligned}
\end{equation}
In fact, with the additional condition $\bar P=0$, we are back to the jump-free case studied in \cite{muellergronbach2019}.
The proof of \eqref{star2} hence works exactly as the one for \cite[(20)]{muellergronbach2019}, which is a part of \cite[Proof of Lemma 5]{muellergronbach2019}.

Using \eqref{star2} we obtain
\begin{equation}\label{star4}
\begin{aligned}
 \P(\mathcal{Z}_k^s\cap \mathcal{Z}_k^t) 
  & =
   \P\bigg(\mathcal{Z}_k^s\cap \mathcal{Z}_k^t \cap \{\bar P=0\}\cap \bigg\{\max_{i\in\{1,2,3\}} |\bar W_i|   \le \sqrt{2\log(T/\delta)}\bigg\}\bigg) 
     \\&\quad+
      \P\bigg( \mathcal{Z}_k^s\cap \mathcal{Z}_k^t\cap\bigg( \{\bar P>0\}\cup \left\{\max_{i\in\{1,2,3\}} |\bar W_i|   > \sqrt{2\log(T/\delta)}\right\} \bigg)\bigg)
    \\& \le
  \P\bigg(\mathcal{Z}_k^s\cap\bigg\{|\appx_{\et-(t-\et)}-\zeta_k| \le \frac{12c_2(1+|\zeta_k|)\cdot(1+|\bar W_1|+|\bar W_2|)}{\sqrt{T/\delta}}\bigg\}\bigg)
  \\&\quad+
      \P\bigg( \mathcal{Z}_k^s\cap\bigg( \{\bar P>0\}\cup \left\{\max_{i\in\{1,2,3\}} |\bar W_i|   > \sqrt{2\log(T/\delta)}\right\} \bigg)\bigg).
\end{aligned}
\end{equation}
For the first term on the right hand side of \eqref{star4}, we use the fact that the sum of standard normally distributed random variables is normally distributed with mean 0 and variance 2.
\begin{equation}\label{star5}
\begin{aligned}
  &\P\bigg(\mathcal{Z}_k^s\cap\bigg\{|\appx_{\et-(t-\et)}-\zeta_k| \le \frac{12c_2(1+|\zeta_k|)\cdot(1+|\bar W_1|+|\bar W_2|)}{\sqrt{T/\delta}}\bigg\}\bigg)
  \\&\quad\le
 \frac{2}{\pi} \int_{[0,\infty)\times [0,\infty)}   \P\bigg(\mathcal{Z}_k^s\cap\bigg\{|\appx_{\et-(t-\et)}-\zeta_k| \le \frac{12c_2(1+|\zeta_k|)\cdot(1+z_1+z_2)}{\sqrt{T/\delta}}\bigg\}\bigg) 
     \\&\qquad \cdot e^{-\frac{(z_1)^2+(z_2)^2}{2}} d(z_1,z_2)
   \\&\quad\le
 \frac{2}{\pi} \int_{\R^2}   \P\bigg(\mathcal{Z}_k^s\cap\bigg\{|\appx_{\et-(t-\et)}-\zeta_k| \le \frac{12\sqrt{2}c_2(1+|\zeta_k|)\cdot(1+|z_1+z_2|/\sqrt{2})}{\sqrt{T/\delta}}\bigg\}\bigg) 
    \\&\qquad\cdot e^{-\frac{(z_1)^2+(z_2)^2}{2}} d(z_1,z_2)
    \\&\quad=
 \frac{4}{\sqrt{2\pi}} \int_{\R}   \P\bigg(\mathcal{Z}_k^s\cap\bigg\{|\appx_{\et-(t-\et)}-\zeta_k| \le \frac{12\sqrt{2}c_2(1+|\zeta_k|)\cdot(1+|z|)}{\sqrt{T/\delta}}\bigg\}\bigg) \cdot e^{-\frac{z^2}{2}} dz.
\end{aligned}
\end{equation}
For the first term on the right hand side of \eqref{star4} we use a standard Gaussian tail estimate.
\begin{equation}\label{star6}
\begin{aligned}
&  \P\!\left( \mathcal{Z}_k^s\cap\left( \{\bar P>0\}\cup \left\{\max_{i\in\{1,2,3\}} |\bar W_i|   > \sqrt{2\log(T/\delta)}\right\} \right)\right)
  \\&\quad=
  \P\!\left( \mathcal{Z}_k^s\cap\left( \{\bar P>0\}\cup \{|\bar W_1|   > \sqrt{2\log(T/\delta)}\} \cup \{|\bar W_2|   > \sqrt{2\log(T/\delta)}\} \cup \{|\bar W_3|   > \sqrt{2\log(T/\delta)}\} \right)\right)
    \\&\quad\le
  \P( \mathcal{Z}_k^s) \cdot \left(  3 \,\P( \{|\bar W_1|   > \sqrt{2\log(T/\delta)}\} ) + \P(\bar P>0) \right)
      \\&\quad\le
  \P( \mathcal{Z}_k^s) \cdot \left(  \frac{3 \delta}{T\sqrt{\pi \log(T/\delta)}}+ \P(\bar P>0) \right)
    =
  \P( \mathcal{Z}_k^s) \cdot \left(  \frac{3 \delta}{T\sqrt{\pi \log(T/\delta)}}+ 1-\P(\bar P=0) \right)
            \\&\quad=
  \P( \mathcal{Z}_k^s) \cdot \left(  \frac{3 \delta}{T\sqrt{\pi \log(T/\delta)}}+ 1-e^{-\int_{\et-\delta}^t \lambda_s ds} \right)
         \le
  \P( \mathcal{Z}_k^s) \cdot \left(  \frac{3 \delta}{T\sqrt{\pi \log(T/\delta)}}+ 2\delta \|\lambda\|_\infty \right).
  \end{aligned}
\end{equation}
Combining \eqref{star4} with \eqref{star5} and \eqref{star6} finishes the proof.
\end{proof}

\begin{lemma}\label{help1-cross}
Let Assumptions \ref{ass:ex-un} hold. Let $s\in[0,T)$.
There exists a constant ${C_2}\in(0,\infty)$ such that for all $k\in\{1,\dots,m\}$, $\delta\in(0,1)$ sufficiently small,
\begin{equation*}
 \int_s^T  \P(\mathcal{Z}_k^s\cap \mathcal{Z}_k^t) dt
\le
{C_2}\delta^{1/2}\cdot \left(\P(\mathcal{Z}_k^s)
+  \E\!\left[\1_{\mathcal{Z}_k^s} \cdot (\appx_{\es+\delta}-\zeta_k)^2 \right]\right).
\end{equation*}
\end{lemma}

\begin{proof}
We follow the first part of \cite[Proof of Lemma 6]{muellergronbach2019}.
For $s\ge T-\delta$,
\begin{equation*}
 \int_s^T  \P(\mathcal{Z}_k^s\cap \mathcal{Z}_k^t) dt
\le
 \int_{T-\delta}^T \P(\mathcal{Z}_k^s) dt=   \P(\mathcal{Z}_k^s)\cdot  \delta.
\end{equation*}
Therefore we may assume $s< T-\delta$ and hence $\es\le T-2\delta$.
With this
\begin{equation}\label{cat1}
\begin{aligned}
& \int_s^T  \P(\mathcal{Z}_k^s\cap \mathcal{Z}_k^t) dt
=
 \int_s^{\es+2\delta}  \P(\mathcal{Z}_k^s\cap \mathcal{Z}_k^t) dt
+
 \int_{\es+2\delta} ^T \P(\mathcal{Z}_k^s\cap \mathcal{Z}_k^t) dt
\\&\quad\le
 \int_\es^{\es+2\delta}  \P(\mathcal{Z}_k^s) dt
+
 \int_{\es+2\delta} ^T \P(\mathcal{Z}_k^s\cap \mathcal{Z}_k^t) dt
\le
2\delta\, \P(\mathcal{Z}_k^s)
+
 \int_{\es+2\delta} ^T \P(\mathcal{Z}_k^s\cap \mathcal{Z}_k^t) dt.
\end{aligned}
\end{equation}
Let $\tilde \ell \in \N$ be such that $t_{\tilde \ell} = \es + 2 \delta$. For $t\in[\es+2\delta,T]$ it holds that $\et\ge \es+2\delta$ and hence $\et-\delta\ge\es+\delta\ge s$. Hence we may apply Lemma \ref{help0-cross} and obtain that there exists a constant $c_1\in(0,\infty)$ such that
\begin{equation*}
\begin{aligned}
& \int_s^T  \P(\mathcal{Z}_k^s\cap \mathcal{Z}_k^t) dt
\\&\quad\le
c_1 \P(\mathcal{Z}_k^s)\delta
+ c_1
 \int_{\es+2\delta} ^T 
\l\int_\R \P\!\left(\mathcal{Z}_k^s\cap  \left\{|\appx_{\et-(t-\et)}-\zeta_k| \le c_1\delta^{1/2} (1+|z|)\right\}\right) \cdot e^{-\frac{z^2}{2}} dz
dt
\\&\quad=
c_1 \P(\mathcal{Z}_k^s)\delta
+ c_1
\sum_{\ell=\tilde \ell}^{N-1} \int_{t_\ell} ^{t_{\ell+1}}
\l\int_\R \P\!\left(\mathcal{Z}_k^s\cap  \left\{|\appx_{t_\ell-(t-t_\ell)}-\zeta_k| \le c_1\delta^{1/2} (1+|z|)\right\}\right) \cdot e^{-\frac{z^2}{2}} dz
dt
.
\end{aligned}
\end{equation*}
Substituting $u=t_\ell-(t-t_\ell)$ gives
\begin{equation*}
\begin{aligned}
& \int_s^T  \P(\mathcal{Z}_k^s\cap \mathcal{Z}_k^t) dt
\\&\quad=
c_1 \P(\mathcal{Z}_k^s)\delta
- c_1
\sum_{\ell=\tilde \ell}^{N-1} \int_{t_\ell} ^{t_{\ell-1}}
\l\int_\R \P\!\left(\mathcal{Z}_k^s\cap  \left\{|\appx_{u}-\zeta_k| \le c_1\delta^{1/2} (1+|z|)\right\}\right) \cdot e^{-\frac{z^2}{2}} dz
du
\\&\quad=
c_1 \P(\mathcal{Z}_k^s)\delta
+ c_1
\sum_{\ell=\tilde \ell}^{N-1} \int_{t_{\ell-1}} ^{t_{\ell}}
\l\int_\R \P\!\left(\mathcal{Z}_k^s\cap  \left\{|\appx_{u}-\zeta_k| \le c_1\delta^{1/2} (1+|z|)\right\}\right) \cdot e^{-\frac{z^2}{2}} dz
du
\\&\quad=
c_1 \P(\mathcal{Z}_k^s)\delta
+ c_1
 \int_{\es+\delta} ^{T-\delta}
\l\int_\R \P\!\left(\mathcal{Z}_k^s\cap  \left\{|\appx_u-\zeta_k| \le c_1\delta^{1/2} (1+|z|)\right\}\right) \cdot e^{-\frac{z^2}{2}} dz
du
\\&\quad=
c_1 \P(\mathcal{Z}_k^s)\delta
+ c_1
\l\int_\R \left[
 \int_{\es+\delta} ^{T-\delta}
\P\!\left(\mathcal{Z}_k^s\cap  \left\{|\appx_u-\zeta_k| \le c_1\delta^{1/2} (1+|z|)\right\}\right)  du \right]e^{-\frac{z^2}{2}} 
dz
.
\end{aligned}
\end{equation*}
The Markov property yields
\begin{equation}\label{cat2}
\begin{aligned}
&
 \int_{\es+\delta} ^{T-\delta}
\P\!\left(\mathcal{Z}_k^s\cap  \left\{|\appx_t-\zeta_k| \le c_1\delta^{1/2} (1+|z|)\right\}\right)  dt
\\&\quad=
\E\!\left[\1_{\mathcal{Z}_k^s} \cdot \E\!\left[  \int_{\es+\delta} ^{T-\delta}  \1_{ \left\{|\appx_t-\zeta_k| \le c_1\delta^{1/2} (1+|z|)\right\}} dt \,\vline\,\appx_{\es+\delta}\right]\right].
\end{aligned}
\end{equation}
Lemma \ref{occup_time} ensures that there exists $c_2\in(0,\infty)$ such that
\begin{equation*}
\begin{aligned}
&
  \E\!\left[  \int_{\es+\delta} ^{T-\delta}  \1_{ \left\{|\appx_t-\zeta_k| \le c_1\delta^{1/2} (1+|z|)\right\}} dt \,\vline\,\appx_{\es+\delta}=x\right]
\\&\quad= 
   \int_{0} ^{T-2\delta-\es}  \E\!\left[ \1_{ \left\{|\appx_t-\zeta_k| \le c_1\delta^{1/2} (1+|z|)\right\}}  \,\vline\,\appx_{0}=x\right] dt
\\&\quad\le
c_2(1+x^2)(c_1\delta^{1/2} (1+|z|)+\delta^{1/2})
.
\end{aligned}
\end{equation*}
Combining this with \eqref{cat2} gives
\begin{equation*}
\begin{aligned}
&
 \int_{\es+\delta} ^{T-\delta}
\P\!\left(\mathcal{Z}_k^s\cap  \left\{|\appx_t-\zeta_k| \le c_1\delta^{1/2} (1+|z|)\right\}\right)  dt
\\&\quad\le
\E\!\left[\1_{\mathcal{Z}_k^s} \cdot c_2(1+(\appx_{\es+\delta})^2)(c_1\delta^{1/2} (1+|z|)+\delta^{1/2}) \right]
\\&\quad\le
2c_2(c_1(1+|z|)+1)(1+|\zeta_k|^2)\delta^{1/2}\cdot
\left(\P(\mathcal{Z}_k^s)+\E\!\left[\1_{\mathcal{Z}_k^s} \cdot (\appx_{\es+\delta}-\zeta_k)^2
\right]\right).
\end{aligned}
\end{equation*}

\end{proof}

\begin{lemma}\label{help3-cross}
Let Assumptions \ref{ass:ex-un} hold.
There exists a constant ${C_3}\in(0,\infty)$ such that for all $k\in\{1,\dots,m\}$,
\begin{equation*}
 \int_0^T \E\!\left[\1_{\mathcal{Z}_k^s} \cdot (\appx_{\es+\delta}-\zeta_k)^2 \right]ds\le {C_3} \delta.
\end{equation*}
\end{lemma}

\begin{proof}
First, note that
\begin{equation*}
\begin{aligned}
\1_{\mathcal{Z}_k^s} \cdot |\appx_{\es+\delta}-\zeta_k| 
&\le 
\1_{\mathcal{Z}_k^s} \cdot (|\appx_{\es+\delta}-\appx_s| + |\appx_s-\zeta_k| )
\\&\le 
\1_{\mathcal{Z}_k^s} \cdot (|\appx_{\es+\delta}-\appx_s| + |\appx_s-\appx_\es| )
.
\end{aligned}
\end{equation*}
With this,
\begin{equation*}
\begin{aligned}
& \int_0^T \E\!\left[\1_{\mathcal{Z}_k^s} \cdot (\appx_{\es+\delta}-\zeta_k)^2 \right]ds
\le
 \int_0^T \E\!\left[(|\appx_{\es+\delta}-\appx_s| + |\appx_s-\appx_\es| )^2 \right]ds
 \\&\quad=
  \sum_{k=0}^{N-1}\int_{t_k}^{t_{k+1}} \E\!\left[(|\appx_{\es+\delta}-\appx_s| + |\appx_s-\appx_\es| )^2 \right]ds
   \\&\quad\le
 2\cdot  \sum_{k=0}^{N-1}\int_{t_k}^{t_{k+1}}\left( \E\!\left[|\appx_{\es+\delta}-\appx_s|^2\right] + \E\!\left[|\appx_s-\appx_\es|^2  \right]\right)ds.
 \end{aligned}
\end{equation*}
Hence, Lemma \ref{msq-appx} yields
\begin{equation*}
\begin{aligned}
& \int_0^T \E\!\left[\1_{\mathcal{Z}_k^s} \cdot (\appx_{\es+\delta}-\zeta_k)^2 \right]ds
\le
 4TC^{(\text{M})}\cdot \delta.
 \end{aligned}
\end{equation*}

\end{proof}

\begin{proposition}\label{prob-cross}
Let Assumptions \ref{ass:ex-un} hold.
There exists a constant $C^{(\text{cross})}\in(0,\infty)$ such that for all $k\in\{1,\dots,m\}$, $\delta\in(0,1)$ sufficiently small,
\begin{equation*}
\E\!\left[\left|\int_0^{T}   \1_{\{(x,y)\in\R^2 \colon (x-\zeta_k)(y-\zeta_k)\le0\}}(\appx_\es,\appx_s) ds\right|^2\right] \le C^{(\text{cross})} \cdot \delta.
\end{equation*}
\end{proposition}

\begin{proof}
By Lemma \ref{help1-cross} we get that
\begin{equation}\label{heinzi}
\begin{aligned}
&\E\!\left[\left|\int_0^{T}   \1_{\{(x,y)\in\R^2 \colon (x-\zeta_k)(y-\zeta_k)\le0\}}(\appx_\es,\appx_s) ds\right|^2\right] 
= 2\cdot\int_0^T \int_s^T  \P(\mathcal{Z}_k^s\cap \mathcal{Z}_k^t) dtds
\\&\quad\le
2\,{C_2}\delta^{1/2}\cdot \left(  \int_0^T\P(\mathcal{Z}_k^s)ds
+  \int_0^T \E\!\left[\1_{\mathcal{Z}_k^s} \cdot (\appx_{\es+\delta}-\zeta_k)^2 \right]ds\right).
\end{aligned}
\end{equation}

Next note that in Lemma \ref{help1-cross} the only requirement on $\mathcal{Z}_k^s$ was that it is $\cF_s$-measurable.
Replacing $\mathcal{Z}_k^s$ by $\Omega$ and setting $s=0$ in Lemma \ref{help1-cross} and applying Lemma \ref{msq-appx} gives
\begin{equation*}
 \int_0^T \P(\mathcal{Z}_k^t) dt\le {C_2}\delta^{1/2}\cdot \left(1
+  \E\!\left[|\appx_{\es+\delta}-\zeta_k|^2 \right]\right)
\le {C_2}\delta^{1/2} \left(1+
2(\zeta_k)^2 +2C^{(\text{M})} \right).
 \end{equation*}

Combining \eqref{heinzi} with this and Lemma \ref{help3-cross} yields
\begin{equation*}
\begin{aligned}
&\E\!\left[\left|\int_0^{T}   \1_{\{(x,y)\in\R^2 \colon (x-\zeta_k)(y-\zeta_k)\le0\}}(\appx_\es,\appx_s) ds\right|^2\right] \\&\quad\le
2\,{C_2}\delta^{1/2}\cdot \left( {C_2} \left(1+
2(\zeta_k)^2 +2C^{(\text{M})} \right) \delta^{1/2}
+ {C_3} \delta\right).
\end{aligned}
\end{equation*}
This closes the proof.
\end{proof}

\subsection{Main result}

\begin{theorem}
Let Assumptions \ref{ass:ex-un} hold.
Then there exists a constant $C^{(\text{EM})}\in(0,\infty)$ such that for all $\delta \in(0,1)$ sufficiently small,
\begin{equation*}
 \bigg(\E\bigg[\sup_{t\in[0,T]}|X_t- \appx_t|^2\bigg]\bigg)^{\!1/2}\le C^{(\text{EM})}\delta^{1/2}.
\end{equation*}
\end{theorem}

\begin{proof}
Let $G$ be as in \eqref{eq:G-1d} and $Z$ be as in \eqref{eq:Z}.
Since $G^{-1}$ is Lipschitz,
\begin{align}\label{eq:est-lip}
\bigg( \E\bigg[\sup_{t\in[0,T]}|X_t- \appx_t|^2\bigg]\bigg)^{\!1/2}
 &=
  \bigg(\E\bigg[\sup_{t\in[0,T]}|G^{-1}(Z_t)- G^{-1}(G(\appx_t))|^2\bigg]\bigg)^{\!1/2}
\\&\le
L_{G^{-1}}  \bigg(\E\bigg[\sup_{t\in[0,T]}|Z_t- G(\appx_t)|^2\bigg]\bigg)^{\!1/2}
\end{align}
and by the triangle inequality,
\begin{align}\label{eq:est-triangle}
 \bigg(\E\bigg[\sup_{t\in[0,T]}|Z_t- G(\appx_t)|^2\bigg]\bigg)^{\!1/2}
&\le
 \bigg(\E\bigg[\sup_{t\in[0,T]}|Z_t- \appz_t|^2\bigg]\bigg)^{\!1/2}
 \\&\quad+
 \bigg(\E\bigg[\sup_{t\in[0,T]}|\appz_t- G(\appx_t)|^2\bigg]\bigg)^{\!1/2}.
\end{align}
There exists a constant $c_1\in(0,\infty)$ such that
\begin{align}\label{eq:est-euler}
 \E\bigg[\sup_{t\in[0,T]}|Z_t- \appz_t|^2\bigg]\le c_1\delta.
\end{align}
Hence our task is to estimate the second term in \eqref{eq:est-triangle}.
For all $\tau\in[0,T]$ let
\begin{equation*}
u(\tau):=\E\bigg[\sup_{t\in[0,\tau]}|\appz_t- G(\appx_t)|^2\bigg]
\end{equation*}
and let for all $x_1,x_2\in\R$, $\nu(x_1,x_2)=G'(x_1)\mu(x_2)+\frac{1}{2} G''(x_1) \sigma^2(x_2)$ and $\varsigma(x_1,x_2)=G(x_1+\rho(x_2))-G(x_1)$.
By the Meyer-It\^o formula, which follows from \cite[p.~221, Theorem 71]{protter2004}, we obtain
\begin{equation*}
G(\appx_{t})=G(\appx_{0})+\int_0^{t}\nu(\appx_s,\appx_{\es}) ds+\int_0^{t}G'(\appx_s)\sigma(\appx_{\es})dW_s+\int_0^{t}\varsigma (\appx_{s},\appx_{\es}) dN_s.
\end{equation*}
This yields that
\begin{align*}
u(\tau)
&=
\E\bigg[\sup_{t\in[0,\tau]}\Big|
\int_0^{t}\tilde \mu(\appz_{\es}) ds+\int_0^{t}\tilde \sigma(\appz_{\es}) dW_s+\int_0^t \tilde \rho(\appz_{\es}) dN_s
\\ & \qquad -
 \int_0^{t} \nu(\appx_s,\appx_{\es})ds - \int_0^{t}G'(\appx_s)\sigma(\appx_{\es})dW_s -  \int_0^{t} \varsigma(\appx_{s},\appx_\es) dN_s
\Big|^2\bigg]
\end{align*}
Using
\begin{equation*}
	\nu  (\appx_{\es},\appx_{\es})=\tilde\mu(G(\appx_{\es}))   \quad \text{ and } \quad \varsigma (\appx_{\es},\appx_{\es})=\tilde\rho(G(\appx_{\es})) 
\end{equation*}
we get
\begin{align}\label{6E-i}
u(\tau) \le 6 \cdot \sum_{i=1}^6   E_i(\tau)
\end{align}
with
\begin{align*}
E_1 (\tau)&=\E\bigg[\sup_{t\in[0,\tau]}  \Big|\int_0^{t}[\tilde \mu(\appz_{\es})-\tilde\mu(G(\appx_{\es}))]  ds \Big|^2  \bigg] ,\\
E_2 (\tau)&=  \E\bigg[\sup_{t\in[0,\tau]}  \Big|\int_0^{t}[\tilde \sigma(\appz_{\es})-\tilde\sigma(G(\appx_{\es}))] dW_s \Big|^2  \bigg],\\
E_3 (\tau)&= \E\bigg[\sup_{t\in[0,\tau]}  \Big|\int_0^{t}[\tilde \rho(\appz_{\es})-\tilde\rho(G(\appx_{\es})) ] dN_s \Big|^2 \bigg],\\
E_4 (\tau)&= \E\bigg[\sup_{t\in[0,\tau]}  \Big|\int_0^{t}[\nu(\appx_{\es},\appx_{\es})-\nu(\appx_s,\appx_{\es}) ] ds \Big|^2 \bigg],\\
E_5(\tau) &= \E\bigg[\sup_{t\in[0,\tau]}  \Big|\int_0^{t}[G'(\appx_{\es})\sigma(\appx_{\es})-G'(\appx_s)\sigma(\appx_{\es})] dW_s \Big|^2 \bigg], \\
E_6(\tau)&= \E\bigg[\sup_{t\in[0,\tau]}   \Big|\int_0^{t}[\varsigma(\appx_{\es},\appx_{\es})-\varsigma(\appx_{s},\appx_{\es}) ] dN_s \Big|^2 \bigg].
\end{align*}
The Cauchy-Schwarz inequality yields
\begin{align*}
E_1(\tau) &\le T \cdot\E\!\left[\int_0^{\tau}|\tilde \mu(\appz_{\es})-\tilde\mu(G(\appx_{\es}))|^2 ds   \right] ,
\end{align*}
and the Burkholder-Davis-Gundy inequality, see, e.g., \cite[Lemma 3.7]{hutzenthaler2012}, yields
\begin{align*}
E_2(\tau) &\le 2 \cdot\E\!\left[\int_0^{\tau}|\tilde \sigma(\appz_{\es})-\tilde\sigma(G(\appx_{\es}))|^2 ds   \right] ,\\
E_5(\tau) &\le 2 \cdot \E\!\left[\int_0^{\tau}|G'(\appx_{\es})\sigma(\appx_{\es})-G'(\appx_s)\sigma(\appx_{\es}) |^2 ds \right].
\end{align*}
Finally, since $N_s=\tilde N_s+\kappa_s$ we get by Doob's maximum inequality, It\^o's isometry, and the Cauchy-Schwarz inequality that
\begin{align*}
E_3(\tau) &\le 2\cdot\E\bigg[\sup_{t\in[0,\tau]}  \Big|\int_0^{t}[\tilde \rho(\appz_{\es})-\tilde\rho(G(\appx_{\es})) ] d\tilde N_s \Big|^2 \bigg]+
2\cdot\E\bigg[\sup_{t\in[0,\tau]}  \Big|\int_0^{t}[\tilde \rho(\appz_{\es})-\tilde\rho(G(\appx_{\es})) ] d\kappa(s) \Big|^2 \bigg] \\
&\le 8\cdot\E\bigg[\Big|\int_0^{\tau}[\tilde \rho(\appz_{\es})-\tilde\rho(G(\appx_{\es})) ] d\tilde N_s \Big|^2\bigg]+
2\cdot\E\bigg[\sup_{t\in[0,\tau]}  \Big|\int_0^{t}[\tilde \rho(\appz_{\es})-\tilde\rho(G(\appx_{\es})) ] \lambda(s) ds \Big|^2 \bigg] \\
&\le 8\cdot\E\bigg[\int_0^{\tau}|\tilde \rho(\appz_{\es})-\tilde\rho(G(\appx_{\es})) |^2\lambda(s)ds\bigg]+2T\cdot\E\bigg[\int_0^{\tau}|\tilde \rho(\appz_{\es})-\tilde\rho(G(\appx_{\es})) |^2
\cdot|\lambda(s)|^2ds\bigg]\\
&\le2\|\lambda\|_{\infty}(4+T\|\lambda\|_{\infty})\cdot\E\bigg[\int_0^{\tau}|\tilde \rho(\appz_{\es})-\tilde\rho(G(\appx_{\es})) |^2 ds\bigg]
\end{align*}
and analogously
\begin{equation*}
E_6(\tau) \le2\|\lambda\|_{\infty}(4+T\|\lambda\|_{\infty})\cdot\E\!\left[\int_0^{\tau}|\varsigma(\appx_{\es},\appx_{\es})-\varsigma(\appx_{s},\appx_{\es})|^2 ds\right].
\end{equation*}

Next, using that $\tilde \mu, \tilde \sigma, \tilde \rho$ are Lipschitz, we get that
\begin{equation}\label{E1E2E3}
\begin{aligned}
E_1(\tau)& \le T (L_{\tilde\mu})^2 \cdot \int_0^{\tau}\E\!\left[|\appz_{\es}-G(\appx_{\es})|^2  \right] ds\le T (L_{\tilde\mu})^2 \cdot \int_0^{\tau}u(s)ds,
\\
E_2(\tau)& \le 2 (L_{\tilde\sigma})^2 \cdot \int_0^{\tau}\E\!\left[|\appz_{\es}-G(\appx_{\es})|^2  \right] ds\le 2 (L_{\tilde\sigma})^2 \cdot \int_0^{\tau}u(s)ds,
\\
E_3(\tau)& \le 2\|\lambda\|_{\infty}(4+T\|\lambda\|_{\infty}) (L_{\tilde\rho})^2 \cdot \int_0^{\tau}\E\!\left[|\appz_{\es}-G(\appx_{\es})|^2  \right] ds\\&\le (8\|\lambda\|_{\infty}+2T\|\lambda\|_{\infty}^2) (L_{\tilde\rho})^2 \cdot \int_0^{\tau}u(s)ds.
\end{aligned}
\end{equation}

The linear growth condition on $\sigma$, the Lipschitz continuity of $G'$, and the fact that $\appx_{\es}$ is $\cF_\es$-measurable give
\begin{align*}
E_5(\tau) & \le 2 (c_\sigma)^2 (L_{G'})^2 \cdot \int_0^{T} \E\!\left[(1+|\appx_{\es}| )^2 \cdot |\appx_{\es}-\appx_s|^2  \right] ds
\\&=
2 (c_\sigma)^2 (L_{G'})^2 \cdot \int_0^{T} \E\!\left[(1+|\appx_{\es}| )^2 \cdot\E\!\left[ |\appx_{\es}-\appx_s|^2  \,\big| \, \cF_\es\right]  \right] ds
.
\end{align*}
Let $c_2=\|\lambda\|_{\infty}\cdot \max\{\|\lambda\|_{\infty},1\}$.
Since $W_s-W_\es$ and $N_s-N_\es$ are independent of $\cF_\es$, $\appx_\es$ is $\cF_\es$-measurable, and by the linear growth condition on $\mu$, $\sigma$, and $\rho$ we get for $s\in [0,T]$,
\begin{equation}
\begin{aligned}
	&\E\!\left[ |\appx_{\es}-\appx_s|^2 \, \big| \, \cF_\es \right]=\E\!\left[|\mu(\appx_\es)(s-\es)+\sigma(\appx_\es)(W_s-W_{\es})+\rho(\appx_\es)(N_s-N_{\es})|^2 \, \big| \, \cF_\es\right]
	\\\qquad&\le 
	3|\mu(\appx_\es)|^2 (s-\es)^2 +3|\sigma(\appx_\es)|^2 (s-\es)+3|\rho(\appx_\es)|^2 (\|\lambda\|_{\infty}^2(s-\es)^2+\|\lambda\|_{\infty} (s-\es))
	\\\qquad&\le
	6((c_\mu)^2+(c_\sigma)^2+2c_2(c_\rho)^2 )(1+|\appx_{\es}|^2)\cdot |s-\es|.
\end{aligned}
\end{equation}
This and Lemma \ref{msq-appx} yield
\begin{equation}\label{E5}
\begin{aligned}
E_5(\tau)&\le 
48 (c_\sigma)^2 (L_{G'})^2 ((c_\mu)^2+(c_\sigma)^2+2(c_\rho)^2\|\lambda\|_{\infty}^2 ) \cdot \int_0^{T}  |s-\es|\cdot\left(1+\E\!\left[|\appx_{\es}| ^4 \right] \right)ds
\\&\le 48 T (c_\sigma)^2 (L_{G'})^2 ((c_\mu)^2+(c_\sigma)^2+2c_2(c_\rho)^2 \|\lambda\|_{\infty}^2)(1+C^{(\text{M})})\cdot \delta.
\end{aligned}
\end{equation}

The Lipschitz continuity of $G$ establishes
\begin{align*}
E_6(\tau)
&\le 4\|\lambda\|_{\infty}(4+T\|\lambda\|_{\infty}) 
\\&\quad\cdot\int_0^{T}\E\!\left[|G(\appx_{\es}+\rho(\appx_{\es}))-G(\appx_{s}+\rho(\appx_{\es}))|^2+|G(\appx_{s})-G(\appx_{\es})|^2 \right]ds
\\&\le 8 (L_G)^2\|\lambda\|_{\infty}(4+T\|\lambda\|_{\infty}) \cdot\int_0^{T}\E\!\left[|\appx_{s}-\appx_{\es}|^2\right] ds.
\end{align*}
By  Lemma \ref{msq-appx},
\begin{align}\label{E6}
E_6(\tau)
&=8 C^{(\text{M})} (L_G)^2\|\lambda\|_{\infty}(4+T\|\lambda\|_{\infty})  \cdot  \int_0^{T} |s-\es| ds
\\&=  8 C^{(\text{M})}T(L_G)^2\|\lambda\|_{\infty}(4+T\|\lambda\|_{\infty})  \cdot \delta.
\end{align}
For estimating $E_4(\tau)$ first note that
\begin{equation}\label{E4-0}
\begin{aligned}
E_4 (\tau)&\le 2T \cdot \E\bigg[\int_0^{T} |G'(\appx_\es) -G'(\appx_s)|^2 \cdot |\mu(\appx_\es)|^2 ds \bigg]
\\&\quad+
 \frac{1}{2} \cdot \E\bigg[\Big|\int_0^{T} G''(\appx_\es)\sigma^2(\appx_\es)-G''(\appx_s)\sigma^2(\appx_\es) ds \Big|^2 \bigg].
\end{aligned}
\end{equation}
Analog to the estimate of $E_5(\tau)$ above we obtain for the first term in \eqref{E4-0},
\begin{align}\label{E4-1}
&2T \cdot \E\bigg[\int_0^{T} |G'(\appx_\es) -G'(\appx_s)|^2 \cdot |\mu(\appx_\es)|^2 ds \bigg]
\\&\quad\le 48T^2 (c_\mu)^2 (L_{G'})^2 ((c_\mu)^2+(c_\sigma)^2+2c_2(c_\rho)^2)(1+C^{(\text{M})})\cdot \delta.
\end{align}
The second term in \eqref{E4-0} will be analysed in the way introduced in \cite{muellergronbach2019}, but adapted to our setup.
It is essential that we have not applied the Cauchy-Schwarz inequality to the second term in \eqref{E4-0} for recovering the optimal convergence rate $1/2$.
For this we define for all $k\in\{1,\dots,m\}$ the sets
\begin{equation*}
\mathcal{Z}_k=\{(x,y)\in\R^2 \colon (x-\zeta_k)(y-\zeta_k)\le0\},\qquad \mathcal{Z}=\bigcup_{\ell=1}^m \mathcal{Z}_\ell.
\end{equation*}
that is the set of all $x,y\in\R$ which lie on different sides of a point of discontinuity of the drift.
So for $(\appx_\es,\appx_s)\in\mathcal{Z}^c$ we have that a Lipschitz condition is satisfied by $G''$.
So for
\begin{equation}\label{E4-2-0}
\begin{aligned}
&\frac{1}{2} \cdot \E\bigg[\Big|\int_0^{T} G''(\appx_\es)\sigma^2(\appx_\es)-G''(\appx_s)\sigma^2(\appx_\es) ds \Big|^2 \bigg]
 \\&\quad\le
  \E\bigg[\Big|\int_0^{T} \left[G''(\appx_\es)\sigma^2(\appx_\es)-G''(\appx_s)\sigma^2(\appx_\es) \right]\cdot\1_{\mathcal{Z}}(\appx_\es,\appx_s)ds\Big|^2 \bigg]
   \\&\qquad+
    T\cdot\E\bigg[\int_0^{T} |G''(\appx_\es)\sigma^2(\appx_\es)-G''(\appx_s)\sigma^2(\appx_\es)|^2 \cdot\1_{\mathcal{Z}^c}(\appx_\es,\appx_s)ds \bigg]
\end{aligned}
\end{equation}
we can estimate the second term similar to above, that is
\begin{equation}\label{E4-2-1}
\begin{aligned}
&  T\cdot\E\!\left[\int_0^{T} |G''(\appx_\es)\sigma^2(\appx_\es)-G''(\appx_s)\sigma^2(\appx_\es)|^2 \cdot\1_{\mathcal{Z}^c}(\appx_\es,\appx_s) ds\right]
 \\&\quad\le
  T(c_\sigma)^4 (L_{G''})^2\cdot\E\!\left[\int_0^{T}(1+|\appx_\es|)^4\cdot |\appx_\es-\appx_s|^2  ds\right]
   \\&\quad\le 96T^2(c_\sigma)^4 (L_{G''})^2((c_\mu)^2+(c_\sigma)^2+2c_2(c_\rho)^2 )(1+C^{(\text{M})})\cdot \delta
  .
\end{aligned}
\end{equation}
For the first term in \eqref{E4-2-0}, observe that for all $k\in\{1,\dots,m\}$, $x,y \in \R$,
\begin{equation*}
|x| \1_{\mathcal{Z}_k}(x,y)\le (|\zeta_k|+|x-\zeta_k|) \1_{\mathcal{Z}_k}(x,y)\le (|\zeta_k|+|x-y|) \1_{\mathcal{Z}_k}(x,y),
\end{equation*}
and hence
\begin{equation*}
 (1+x^2) \cdot\1_{\mathcal{Z}}(x,y)\le  2m|x-y|^2 + (1+2\max\{|\zeta_1|,|\zeta_m|\}^2) \cdot \sum_{k=1}^m  \1_{\mathcal{Z}_k}(x,y).
\end{equation*}
This and Lemma \ref{msq-appx} yield for the first term in \eqref{E4-2-0},
\begin{equation}\label{E4-2-2-1}
\begin{aligned}
&  \E\bigg[\Big|\int_0^{T} \left[G''(\appx_\es)\sigma^2(\appx_\es)-G''(\appx_s)\sigma^2(\appx_\es)\right] \cdot\1_{\mathcal{Z}}(\appx_\es,\appx_s)ds\Big|^2 \bigg]
    \\&\quad\le
  4 (c_\sigma)^4 \|G''\|_\infty ^2\cdot\E\bigg[\Big|\int_0^{T} (1+|\appx_\es|^2)\cdot\1_{\mathcal{Z}}(\appx_\es,\appx_s)ds\Big|^2 \bigg]
    \\&\quad\le
    4 (c_\sigma)^4 \|G''\|_\infty ^2\cdot\E\bigg[\Big|\int_0^{T} \Big[2m|\appx_\es-\appx_s|^2 
    \\&\qquad+ (1+2\max\{|\zeta_1|,|\zeta_m|\}^2) \cdot \sum_{k=1}^m  \1_{\mathcal{Z}_k}(\appx_\es,\appx_s)\Big]ds\Big|^2 \bigg]  
        \\&\quad\le
    32Tm (c_\sigma)^4 \|G''\|_\infty ^2\cdot \int_0^T \E[|\appx_\es-\appx_s|^4] ds
            \\&\qquad+
              8(c_\sigma)^4(1+2\max\{|\zeta_1|,|\zeta_m|\}^2)^2  \|G''\|_\infty ^2 \cdot\E\bigg[\Big|\int_0^{T}  \sum_{k=1}^m  \1_{\mathcal{Z}_k}(\appx_\es,\appx_s) ds\Big|^2\bigg]
             \\&\quad\le
    32Tm (c_\sigma)^4 \|G''\|_\infty ^2\cdot \sum_{k=0}^{N-1} \int_{t_k}^{t_{k+1}} \E[|\appx_\es-\appx_s|^4] ds
            \\&\qquad+
              8m(c_\sigma)^4(1+2\max\{|\zeta_1|,|\zeta_m|\}^2)^2  \|G''\|_\infty ^2 \cdot\sum_{k=1}^m \E\bigg[\Big|\int_0^{T}   \1_{\mathcal{Z}_k}(\appx_\es,\appx_s) ds\Big|^2\bigg]
       \\&\quad\le
    32T^2 mC^{(\text{M})} (c_\sigma)^4 \|G''\|_\infty ^2\cdot \delta
         \\&\qquad+
              8m^2(c_\sigma)^4(1+2\max\{|\zeta_1|,|\zeta_m|\}^2)^2  \|G''\|_\infty ^2 \cdot \E\bigg[\Big|\int_0^{T}   \1_{\mathcal{Z}_k}(\appx_\es,\appx_s) ds\Big|^2\bigg]
  .
\end{aligned}
\end{equation}
Proposition \ref{prob-cross} yields that 
\begin{equation}\label{E4-2-2-2}
\begin{aligned}
 \E\bigg[\Big|\int_0^{T}   \1_{\mathcal{Z}_k}(\appx_\es,\appx_s) ds\Big|^2\bigg]
\le C^{(\text{cross})}\cdot \delta
  .
\end{aligned}
\end{equation}
Combining \eqref{E4-0} with \eqref{E4-1}, \eqref{E4-2-0}, \eqref{E4-2-1}, \eqref{E4-2-2-1}, and \eqref{E4-2-2-2} we get
\begin{equation}\label{E4}
\begin{aligned}
E_4(\tau)&\le 
 \delta \cdot
\big[48T^2\big((c_\mu)^2 (L_{G'})^2 +2(c_\sigma)^4 (L_{G''})^2\big)((c_\mu)^2+(c_\sigma)^2+2c_2(c_\rho)^2 )(1+C^{(\text{M})})
 \\&\qquad+
     32T^2 mC^{(\text{M})} (c_\sigma)^4 \|G''\|_\infty ^2
 \\&\qquad+
   8m^2C^{(\text{cross})}(c_\sigma)^4(1+2\max\{|\zeta_1|,|\zeta_m|\}^2)^2  \|G''\|_\infty ^2\big] 
   .
\end{aligned}
\end{equation}
Combining \eqref{6E-i} with the estimates \eqref{E1E2E3}, \eqref{E5}, \eqref{E6}, and \eqref{E4} shows that
there exist constants $c_3,c_4\in(0,\infty)$ such that
\begin{equation*}
0\le u(\tau) \le c_4 \int_0^\tau u(s) ds + c_3 \delta.
\end{equation*}
Applying Gronwall's inequality yields for all $\tau \in [0,T]$,
\begin{equation*}
u(\tau)\le c_{3} \exp(c_4 \tau) \cdot\delta\le c_{3} \exp(c_4 T) \cdot\delta.
\end{equation*}
Combining this with \eqref{eq:est-lip}, \eqref{eq:est-triangle}, and \eqref{eq:est-euler} finally yields
\begin{equation*}
\bigg(\E\bigg[\sup_{t\in[0,T]}|X_t- \appx_t|^2\bigg]\bigg)^{\!1/2}
 \le L_{G^{-1}}(c_1 \delta)^{1/2} + L_{G^{-1}}\left( c_{3} \exp(c_4 T) \cdot\delta\right)^{\!1/2}.
 \end{equation*}
This closes the proof.
\end{proof}

\subsection{Numerical examples}

Finally, we do a numerical test to check whether the obtained convergence order $\delta^{1/2}$ can actually be observed. We choose the following coefficients:
\begin{example}\label{ex1}
\begin{align*}
\mu(x)=\begin{cases}
		1 & x\ge 0\\
		-1 & x<0
		\end{cases},
		\qquad
		\sigma(x)=x+0.1,
		\qquad
		\rho(x)=x,
\end{align*}
\end{example}
\begin{example}\label{ex2}
\begin{align*}
\mu(x)=\begin{cases}
		-1 & x\ge 0\\
		1 & x<0
		\end{cases},
		\qquad
		\sigma(x)=x+0.1,
		\qquad
		\rho(x)=x,
\end{align*}
\end{example}
and we set the parameter $\lambda\equiv1$ and the initial value $\xi=0.1$.
In Example \ref{ex1} the drift is outward trending -- away from the discontinuity. In Example \ref{ex2} the drift points towards the discontinuity. Note that both examples satisfy Assumption \ref{ass:ex-un}.
The $L^2$-error is estimated by
\begin{align*}
\text{error}(k) =  \text{mean}\big(|X_T^{(k)}(\omega) - X_T^{(k-1)}(\omega)|^2\big)^{1/2},
\end{align*}
where $X_T^{(k)}(\omega)$ is an approximation of $X_T(\omega)$ with step size $\delta^{(k)}$; the mean is taken over $2^{16}$ sample paths $\omega$.
Figure \ref{err} shows $\log_2(\text{error}(k))$ plotted over $\log_2(\delta ^{(k)})$ for both examples. For Example \ref{ex1} we observe that the slope of the estimated error is approximately the same as the slope of $\log_2(\delta^{1/2})$, which shows that the theoretical convergence order is attained. For Example \ref{ex2} we observe that the slope seems to be even steeper than the slope of $\log_2(\delta^{1/2})$. The same behaviour has been observed and discussed in~\cite{lux} in the Brownian additive noise case. This shows that numerical tests have to be interpreted carefully in the case of a discontinuous drift coefficient and knowing the theoretical rates is even more important.
\begin{center}
\begin{figure}[ht]
\includegraphics[width=\textwidth]{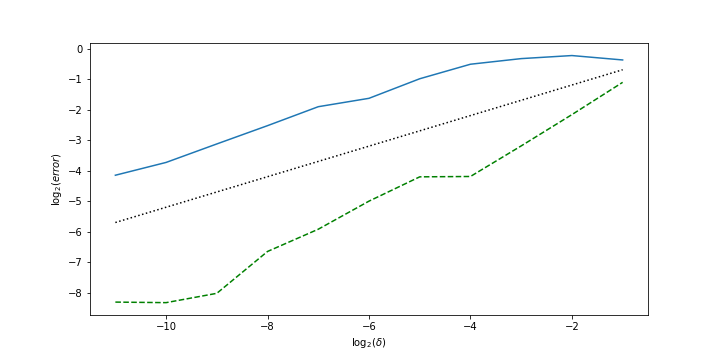}
\caption{The slope of the black dotted line indicates the theoretical convergence order $\delta^{1/2}$; the blue line shows the error estimate for Example \ref{ex1}; the green dashed line shows the error estimate for Example \ref{ex2}.}\label{err}
\end{figure}
\end{center}


\section*{Acknowledgements}

M.~Sz\"olgyenyi is supported by the AXA Research Fund grant ``Numerical Methods for Stochastic Differential Equations with Irregular Coefficients with Applications in Risk Theory and Mathematical Finance''.

P.~Przyby{\l}owicz is supported  by  the  National  Science  Centre, Poland, under project\\ 2017/25/B/ST1/00945.



\vspace{2em}
\centerline{\underline{\hspace*{16cm}}}

 \noindent Pawe{\l} Przyby{\l}owicz \Letter \\
Faculty of Applied Mathematics, AGH University of Science and Technology, Al.~Mickiewicza 30, 30-059 Krakow, Poland\\
pprzybyl@agh.edu.pl\\

\noindent Michaela Sz\"olgyenyi \\
Department of Statistics, University of Klagenfurt, Universit\"atsstra\ss{}e 65-67, 9020 Klagenfurt, Austria\\
michaela.szoelgyenyi@aau.at\\


\end{document}